\documentclass[10pt]{amsart}

\usepackage{amsmath, amsthm, amssymb, amsfonts}
\usepackage{hyperref}
\usepackage{graphicx}
\usepackage{mathrsfs}
\usepackage{enumitem}
\usepackage[margin=1in]{geometry}

\newtheorem{theorem}{Theorem}[section]
\newtheorem{lemma}[theorem]{Lemma}

\newtheorem{fact}[theorem]{Fact}
\newtheorem{proposition}[theorem]{Proposition}
\theoremstyle{remark}

\newcommand{\R}{\mathbb{R}}
\newcommand{\N}{\mathbb{N}}

\newcommand{\Q}{\mathbb{Q}}

\title{A Counterexample Regarding C.E.\ Closed Subsets of $[0,1]$ Under Homeomorphisms}
\author{V.~Bosserhoff}
\date{\today}

\begin{document}

\begin{abstract}
We give an example of a computably enumerable closed subset of $[0,1]$ that is not homeomorphic to 
any computably compact space. This answers a question of Koh, Melnikov and Ng \cite{counterex}.
\end{abstract}

\maketitle

\section{Introduction}

A \emph{computable Polish space} is a Polish space $X$ with a compatible metric $d$ and a dense sequence $(x_i)_i$
such that $(i,j) \mapsto d(x_i, x_j)$ is a computable real function. 
The $x_i$ are called the \emph{special points} of $X$. 
For $x \in X$ and $r \in \R^{>0}$ let $B(x,r)$ be the open ball with center $x$ and radius $r$. Call an open ball
\emph{basic} if the center is a special point and the radius is a rational number.
If there is a computable function $f : \N \to \N^{<\omega}$ such that for every $n$
\[
f(n) = (i_1, \ldots, i_k) \ \Rightarrow \ X = B(x_{i_1}, 2^{-n}) \cup \cdots \cup B(x_{i_k}, 2^{-n}),
\]
the computable Polish space $X$ is called \emph{computably compact}. 
See \cite{compcomp} (and also \cite{compstruc}) for an extensive survey of computable and computably compact 
Polish spaces.

A closed subset $K$ of $X$ is \emph{computably enumerable (c.e.) closed} if the set of basic balls 
that intersect $K$ is computably enumerable (w.r.t.\ some standard encoding). For nonempty $K$, this is 
equivalent to the existence of a computable sequence which is dense in $K$. Note that a c.e.\ closed subset can 
itself be seen as a computable Polish space. $K$ is called 
\emph{computably co-enumerable (co-c.e.) closed} if there is a computable sequence of basic balls
that exhausts $X\setminus K$. 
Call $K$ \emph{computably closed} if it is both c.e.\ closed and co-c.e.\ closed. 
Call $K$ \emph{computably compact} if it is compact, c.e.\ closed, and the set of all finite
open covers by basic balls is computably enumerable. This is equivalent to $K$ being compact
and the set of all finite open covers by basic balls that intersect $K$ being computably enumerable.
In Euclidean space, it is equivalent to $K$ being bounded and computably closed. 
See \cite{brattkapresser} for a comparison of these and further notions of effectivity of
subsets of computable metric spaces. 

There has been some interest in examples of effective spaces and subsets
with a lack of computability that is topologically intrinsic, in the sense of persisting under all -- not necessarily computable --
homeomorphisms. For every $n \geq 1$, the article \cite{bosshert} gives examples of c.e.\ closed and co-c.e.\ closed 
subsets of $[0,1]^n$ 
the image of which under any homeomorphism of $\R^n$ is not a computable (and hence not computably 
compact). For the argument in \cite{bosshert}, it is essential that the map is a 
homeomorphism of all of $\R^n$, so it is not suitable to construct an example of a compact computable Polish \emph{space}
that is not homeomorphic to any computably compact Polish space. 
Such examples were later constructed in \cite{abelian,spectra,compcomp} with proofs based on
methods from algebraic topology and topological group theory. These works contribute to a program of separating
various notions of effectivity on topological spaces; see also
\cite{bazhenov,stone,separating,counterex,benshahar,compstruc}. 
In view of the rather advanced topological methods applied in the cited works,
\cite[Question~4.56]{compcomp} asks for a more elementary construction. 
A topologically tame example $K \subseteq [0, 1]^2$ has recently been constructed 
in \cite{counterex}. The proof technique of \cite{counterex}, however, is
not applicable in $\R$. 
The authors thus ask for a corresponding example $K \subseteq [0, 1]$.
The purpose of the present note is to show that in fact:
\begin{theorem}
\label{main}
There exists a perfect c.e.\ closed $K \subseteq [0,1]$ that is not homeomorphic to any computably compact Polish space.
\end{theorem}

The proof uses the Feiner complexity classes which we will review in Section \ref{feiner}.
Other than that, we only use elementary methods from topology and computability theory. 

\section{The Feiner classes}
\label{feiner}

Let $c : \N \to \N \setminus \{0\}$ be a total computable function. We say that a set $A\subseteq \N$ is in $\Sigma_c^0$ 
if there is a computably enumerable operator $W_e$ such that
\[
n \in A \ \Leftrightarrow \ n \in W_e^{\emptyset^{(c(n)-1)}}
\]
where $\emptyset^{(n)}$ is the $n$-th Turing jump. Also define $\Pi_c^0 := \{ A ~:~ \N \setminus A \in \Sigma_c^0 \}$ 
and $\Delta_c^0 := \Sigma_c^0 \cap \Pi_c^0$.
For constant $c$, we recover the usual $\Sigma$-, $\Pi$- and $\Delta$-classes of the artithmetical hierarchy \cite{rogers}. 
Feiner \cite{feiner} introduced the class $\Delta_c^0$ for $c$ of the form $n \mapsto an + b$ and used
it to construct a c.e.\ Boolean algebra not isomorphic to any computable Boolean algebra. See also the expositions
and further references in e.g.\ \cite{ashknight}, \cite{compstruc}. 

The Feiner classes have been used indirectly in computable
topological structure theory before, when certain results were obtain by redution to 
results in computable algebraic structure theory. 
In \cite{stone} for example, Feiner's example was used to construct a computable topological compact Polish space 
not homeomorphic to any computably metrized Polish space. 

The classical characterizations of the arithmetical $\Sigma$- and $\Pi$- classes 
in terms of quantified computable predicates generalize to the Feiner classes as the classical constructions 
(see e.g.\ \cite[Theorem~14.VIII]{rogers}) are uniform in the number of quantifiers. 
So if $\langle \cdots \rangle : \N^{<\omega} \to \N$ is a standard bi-computable tupling function 
(see e.g.\ \cite[p.~71]{rogers}) and 
\[
Q_n := \begin{cases}
    ``\forall'' & \textrm{ if $n$ is even,} \\
    ``\exists'' & \textrm{ if $n$ is odd,} \\
\end{cases}
\]
then $A \in \Sigma_c^0$ if and only if there exists a computable predicate $p$ such that
\[
x \in A \ \Leftrightarrow \ (\exists y_1) (\forall y_2) \cdots (Q_{c(x)} y_{c(x)}) \ p(\langle x, y_1, \ldots y_{c(x)} \rangle).
\]
A corresponding characterization holds for the $\Pi$-classes. 

It will be convenient later to also work with the following characterization in terms of the quantifier ``$\exists^\infty$''
(meaning ``there exist infinitely many''): 
$A \in \Pi_{n \mapsto 2n+1}^0$ if and only if there is a computable predicate $p$ such that
\[
n \in A \ \Leftrightarrow \ (\exists^\infty y_n) \cdots (\exists^\infty y_1)(\forall x)\ p(\langle n, y_n, \ldots, y_1, x \rangle).
\]
(For the standard arithmetical class $\Pi_{2n+1}^0$ the corresponding statement is \cite[Theorem~14.XVIII]{rogers}. 
The constructions are uniform in $n$.)

The Feiner classes in fact form a hierarchy. We will only need the following special case:
\begin{lemma}
\label{separate}
There exists an $A \in \Sigma_{n \mapsto 2n + 3}^0 \setminus \Sigma_{n \mapsto 2n + 2}^0$ 
with $\{2n ~:~ n\in \N\} \subseteq A$.
\end{lemma}
\begin{proof}
Choose 
\[
A := \{ 2n+1 ~:~ n \in \N,\,  2n+1 \notin W_n^{\emptyset^{(4n+3)}} \} \cup \{2n ~:~ n\in \N\}.
\]
\end{proof}

\section{Outline of the proof of Theorem \ref{main}}
\label{outline}
We start with an informal description: Let $X$ be a compact Polish space.
The idea is to encode subsets of $\N$ in the Cantor-Bendixson ranks
of points appearing in the space (which naturally leads to the appearance of the Feiner classes). 
The usual Cantor-Bendixson rank, however, is not suitable, basically because identifying isolated points
in computably compact spaces is not easier than in merely computable spaces. Computable compactness does, however, help
when it comes to identifying disjoint clopen subsets, in the sense that one quantifier change can be saved
 (compare Lemmas \ref{p2} and \ref{1proper} below). So we use a variant of the Cantor-Bendixson rank
 where a point has positive rank if it is a ``limit'' of clopen sets instead of isolated points. If we want to
 encode a set of natural numbers in these ranks, we have the problem that obviously the existence of a point
 of rank $n$ implies the existence of points of ranks $1, 2, \ldots, n-1$. So we need to ``mark'' the
 points the rank of which we actually want to include. We do this by placing them on the boundaries of closed
 intervals, or more generally by making them elements of connected components with nonempty interiors. 

For the formal treatment, call a point $x \in X$ \emph{proper} or \emph{$1$-proper} if every neighbourhood
of $x$ is intersected by infinitely many pairwise disjoint clopen subsets of $X$. 
For $n > 1$, call a point \emph{$n$-proper} if
every neighbourhood contains infinitely many $(n-1)$-proper points.
If there is an $n \in \N$ such that $x$ is $n$-proper but not $(n+1)$-proper, call $n$ the \emph{proper rank} of $x$. 

Consider a subset $Y \subseteq X$. If there is an $n$ such that $Y$ contains and $n$-proper point, but no $(n+1)$-proper
point, call $n$ the \emph{proper rank} of $Y$. The proper rank of $Y$ shall be $0$ if $Y$ does not contain any proper points. 

Call a connected component of $X$ \emph{$n$-distinguished} if its proper rank is $n$ and its interior in non-empty.
Define
\[
\rho(X) := \left\{ n \in \N ~:~ \textrm{$X$ has an $n$-distinguished component} \right\}.
\]
Clearly, $\rho$ is a topological invariant, i.e.\ $\rho(X_1) = \rho(X_2)$ if $X_1$ and $X_2$ are
homeomorphic.

Our proof will proceed by showing the following two propositions:

\begin{proposition}
\label{upperBound}
If $X$ is a computably compact Polish space then $\rho(X) \in \Sigma_{m \mapsto 2m + 2}^0$.
\end{proposition}

\begin{proposition}
\label{lowerBound}
For every $A \in \Sigma_{m \mapsto 2m + 3}^0$ with $\{2n ~:~ n\in \N\} \subseteq A$,
there is a perfect c.e.\ closed subset $K$ of $[0,1]$ such that $\rho(K) = A$. 
\end{proposition}

For the proof of Theorem \ref{main} it is then sufficient to choose $A$ as in Lemma \ref{separate} and then
$K$ as in Proposition \ref{lowerBound}.

\section{Proof of Proposition \ref{upperBound}}

Throughout this section, we assume that $(X, d)$ is a computable Polish space with special points $(x_i)_i$. 

Call balls $B(x,r)$, $B(y,s)$ 
\emph{formally disjoint} if $d(x,y) > r+s$. Similarly, call $B(x,r)$ \emph{formally enclosed} in $B(y,s)$ if $y \in B(x,r)$ and $d(x,y) + r < s$.

We will repeatedly use the following 
\begin{fact}
\label{decomp}
If $X$ is computably compact, there is a computable enumeration of all decompositions of $X$ into finitely many disjoint clopen subsets.
The clopen sets in this enumeration are given as (encodings of) finite unions of basic balls. We furthermore have:
\begin{itemize}
\item[(i)] In each of the enumerated decompositions, the encoded balls contributing to different clopen sets are pairwise formally disjoint. 
\item[(ii)] For every $\varepsilon > 0$ every clopen decomposition appears in the enumeration in a representation with the radii of all encoded
      balls smaller than $\varepsilon$. (This means every decomposition will appear again and again with higher and higher
      resolution.)
\end{itemize}
\end{fact}

Fact \ref{decomp} is proved in \cite[Lemma~4.21, Thm.~1.1(ii)]{compcomp} for splits of $X$ into two clopen subsets, 
but the argument easily extends to finite decompositions and also yields the additional features (i) and (ii).

\begin{lemma}
\label{infinitely}
Consider the following statements for a ball $B(x,r)$:
\begin{itemize}
    \item[(i)] $B(x,r)$ contains a proper point.
    \item[(ii)] $B(x,r)$ is intersected by infinitely many disjoint clopen subsets of $X$.
    \item[(iii)] For every $k \geq 1$ there are $k$ pairwise disjoint clopen subsets of $X$ that intersect $B(x,r)$.
    \item[(iv)] $\overline{B(x,r)}$ contains a proper point.
\end{itemize}
Then (i) $\Rightarrow$ (ii) $\Leftrightarrow$ (iii) $\Rightarrow$ (iv).
\end{lemma}
\begin{proof}
``(i)$\Rightarrow$(ii)'': By definition.

``(ii)$\Rightarrow$(iii)'': Obvious.

``(ii)$\Rightarrow$(iv)'': 
Let $\{ C_i \}_{i \in I}$ by an infinite family of pairwise disjoint clopen sets that intersect $B(x,r)$: 
Choosing one point $y_i$ from each $C_i \cap B(x,r)$, the family $(y_i)_{i \in I}$
must, by compactness, have a limit point $z \in \overline{B(x,r)}$. Then $z$ is proper by construction. 

``(iii)$\Rightarrow$(ii)'': By assumption, there is a double sequence $(C_{k,i} ~:~ k \geq 1,\ 1\leq i \leq k)$ of clopen sets intersecting 
$B(x,r)$ and having $C_{k,i} \cap C_{k,j} = \emptyset$ for any $1\leq i < j \leq k$. 
For every $k$, define \[ C_k := \{ C_{j,i} ~:~ 1 \leq j \leq k,\, 1 \leq i \leq j \}. \] 
Recall that the clopen subsets form a Boolean algebra under the set operations. 
Let $\langle C_k \rangle$ be the Boolean subalgebra generated by $C_k$. 
Define $\tilde{C}_k$ as the set of all atoms of $\langle C_k \rangle$ that are subsets of $\bigcup C_k$ 
and intersect $B(x,r)$. Note that $|\tilde{C}_k| \geq k$. Furthermore, each $\tilde{C}_{k+1}$ is obtained from $\tilde{C}_k$
by adding some sets that are disjoint from $\bigcup \tilde{C}_k$ and by splitting some elements of $\tilde{C}_k$. 
So if we consider the infinite set $\tilde{C} := \bigcup_k \tilde{C}_k$ under the partial order $\supseteq$, 
every element is either maximal or has at least two incomparable successors. This easily implies that $\tilde{C}$
contains an infinite antichain. The elements of this antichain are the desired infinitely many disjoint clopen
sets intersecting $B(x,r)$. 
\end{proof}

\begin{lemma}
\label{useCompact}
Suppose $X$ is computably compact. Then
\[
\left\{ (n, r, k) \in \N \times \Q^{>0} \times \N ~:~  B(x_n, r) \textrm{ intersects at least $k$ disjoint clopen sets} \right\} \in \Sigma_1^0.
\]
\end{lemma}
\begin{proof}
A basic ball $B(x_n, r)$ intersects $k$ disjoint clopen sets iff there is a decomposition of $X$ into $k$ clopen sets
each of which intersects $B(x_n, r)$. (In fact, if $C_1, \ldots, C_k$ are disjoint clopen sets intersecting $B(x_n, r)$ and $C := C_1 \cup \cdots \cup C_k$,
consider $C_1, \ldots, C_k \cup (X \setminus C)$.) The intersection of $B(x_n, r)$ with any clopen set is open; it is thus nonempty iff it contains a special point.
So $B(x_n, r)$ intersects $k$ disjoint clopen sets iff there is decomposition of $X$ into clopen $C_1, \ldots, C_k$ along with sepcial points
$x_{n_1}, \ldots, x_{n_k}$ with $x_{n_i} \in C_i \cap B(x_n, r)$. The lemma follows because we can enumerate all decompositions
of $X$ into $k$ clopen sets by Fact \ref{decomp}, and for each of these decompositions, we can enumerate the $k$-tuples of special points
where each point is in a different element of the decomposition and all points are in $B(x_n, r)$.
\end{proof}

\begin{lemma}
\label{p2}
Suppose $X$ is computably compact. There is a set $A \in \Pi_2^0$ such that for all $n \in \N$, $r \in \Q^{>0}$
\[
B(x_n, r) \textrm{ contains a proper point} \ \Rightarrow \ \langle n, r \rangle \in A \ \Rightarrow \ \overline{B(x_n, r)} \textrm{ contains a proper point}
\]
\end{lemma}
\begin{proof}
Combine Lemma \ref{infinitely}, implications (i)$\Rightarrow$(iii)$\Rightarrow$(iv), and Lemma \ref{useCompact}.
\end{proof}

\begin{lemma}
\label{infinitely2}
Consider the following statements for a ball $B(x,r)$ and $n>1$:
\begin{itemize}
    \item[(i)] $B(x,r)$ contains an $n$-proper point.
    \item[(ii)] $B(x,r)$ contains infinitely many $(n-1)$-proper points.
    \item[(iii)] For every $k$ there are $k$ formally disjoint basic
                 open balls formally enclosed in $B(x,r)$ each of which contains an $(n-1)$-proper point. 
    \item[(iv)] For every $k$ there are $k$ formally disjoint basic
                 closed balls formally enclosed in $B(x,r)$ each of which contains an $(n-1)$-proper point. 
    \item[(v)] $\overline{B(x,r)}$ contains an $n$-proper point.
\end{itemize}
Then (i) $\Rightarrow$ (ii) $\Leftrightarrow$ (iii) $\Leftrightarrow$ (iv) $\Rightarrow$ (v).
\end{lemma}
\begin{proof}
``(i)$\Rightarrow$(ii)'': By definition.

``(ii)$\Rightarrow$(iii)'' and ``(ii)$\Rightarrow$(iv)'': Obvious.

``(ii)$\Rightarrow$(v)'': By compactness, the $(n-1)$-proper points must have a limit point $z$ in $\overline{B(x,r)}$. 
Then $z$ is $n$-proper.

``(iii)$\Rightarrow$(ii)'' and ``(iv)$\Rightarrow$(ii)'': Easily proved by contradiction. 
\end{proof}

\begin{lemma}
\label{checkBall}
Suppose $X$ is computably compact. There exists a set $A \in \Pi_{\langle m, r, n \rangle \mapsto 2n}^0$ such that for $m \in \N$, $n\geq 1$, $r \in \Q^{>0}$:
\[
B(x_m, r) \textrm{ contains an $n$-proper point} \ \Rightarrow \ \langle m, r, n \rangle \in A \ \Rightarrow \ \overline{B(x_m, r)} \textrm{ contains an $n$-proper point}.
\]
\end{lemma}
\begin{proof}
We need to describe how to compute a predicate $p_3$ such that the set $A$ with
\[
\langle m, r, n \rangle \in A \ \Leftrightarrow \ (\forall k_n)(\exists \ell_n) \cdots (\forall k_1)(\exists \ell_1) \ p_3(\langle m, r, n, k_n, \ell_n, \ldots, k_1, \ell_1 \rangle)
\]
has the desired properties.
By Lemma \ref{p2} there is a computable predicate $p_2$ such that 
\begin{equation}
\label{p2base}
B(x_m, r) \textrm{ contains a proper point} \ \Rightarrow \ (\forall k)(\exists \ell)\ p_2(\langle m, r, k, \ell \rangle) \ \Rightarrow \ \overline{B(x_m, r)} \textrm{ contains a proper point}
\end{equation}
for any $m \in \N$, $r \in \Q^{>0}$.
We obtain $p_3$ by a straighforward recursive construction over $p_2$ using the characterization from Lemma \ref{infinitely2}:
Let some input $\langle m, r, n, k_n, \ell_n, \ldots, k_1, \ell_1 \rangle$ be given. 
In case $n=1$, return $p_2(\langle m, r, k_1, \ell_1 \rangle)$. 
In case $n>1$, we make use of a computable function $g$ such that $(g(m, r, k,\ell))_\ell$ is an enumeration of 
(encodings of) all $k$-tuples of formally disjoint basic balls formally enclosed in $B(x_m,r)$. Compute
\[
g(m, r, k_n,\ell_n) =: \left( B(x_{i_j}, t_j) \right)_{1\leq j \leq k_n}
\]
and return the conjunction 
\[
\bigwedge_{1\leq j \leq k_n} p_3(\langle i_j, t_j, n-1, k_{n-1}, \ell_{n-1,j}, \ldots, k_1, \ell_{1,j}  \rangle),
\]
where we have split $\ell_{v}$ into $k_n$ variables via the decoding $\langle \ell_{v,1}, \ldots,  \ell_{v,k_n} \rangle_{k_n} = \ell_{v}$
for $v = n-1, \ldots, 1$. 

The verification is by induction on $n$. The base case $n=1$ is yielded directly by \eqref{p2base}. For $n>1$:
\begin{equation*}
\begin{array}{rl}
& B(x_m, r) \textrm{ contains an $n$-proper point} \\
& \Rightarrow \ (\forall k_n)(\exists \ell_n) [ g(m, r, k_n, \ell_n) = \left( B(x_{i_j}, t_j) \right)_{1\leq j \leq k_n} \\ & \qquad\qquad\qquad\qquad \Rightarrow \, \bigwedge_{1\leq j \leq k_n} (\textrm{$B(x_{i_j}, t_j)$ contains an $(n-1)$-proper point}) ] \\
& \Rightarrow \ (\forall k_n)(\exists \ell_n) [ g(m, r, k_n, \ell_n) = \left( B(x_{i_j}, t_j) \right)_{1\leq j \leq k_n} \\ & \qquad\qquad\qquad\qquad \Rightarrow \, \bigwedge_{1\leq j \leq k_n} (\forall k_{n-1})(\exists \ell_{n-1}) \cdots (\forall k_1)(\exists \ell_1) \ p_3(\langle i_j, t_j, n-1, k_{n-1}, \ell_{n-1}, \ldots, k_1, \ell_1 \rangle) ] \\
& \Rightarrow \ (\forall k_n)(\exists \ell_n) [ g(m, r, k_n, \ell_n) = \left( B(x_{i_j}, t_j) \right)_{1\leq j \leq k_n} \\ & \qquad\qquad\qquad\qquad \Rightarrow \, \bigwedge_{1\leq j \leq k_n} (\textrm{$\overline{B(x_{i_j}, t_j)}$ contains an $(n-1)$-proper point}) ] \\
& \Rightarrow \ \overline{B(x_m, r)} \textrm{ contains an $n$-proper point}
\end{array}
\end{equation*}
The first implication is Lemma \ref{infinitely2}(i)$\Rightarrow$(iii), the second and third are the induction hypothesis, the fourth is 
Lemma \ref{infinitely2}(iv)$\Rightarrow$(v). To see that the third statement is in fact equivalent to
\[
(\forall k_n)(\exists \ell_n) \cdots (\forall k_1)(\exists \ell_1) \ p_3(\langle m, r, n, k_n, \ell_n, \ldots, k_1, \ell_1 \rangle),
\]
and hence to $\langle m, r, n \rangle \in A$,
move all quantifiers to the left using the Tarski-Kuratowski algorithm \cite{rogers} and compare to the recursive definition of $p_3$. 
\end{proof}

Recall that the \emph{quasi-component} of a point is the intersection of all clopen sets that contain the point.
The next lemmas make use of the following fact about compact spaces (see e.g.\ \cite[Theorem~6.1.23]{engelking}):

\begin{fact}
\label{quasi}
The quasi-component of any $x \in X$ is equal to its connected component. 
\end{fact}

\begin{lemma}
\label{finiteIntersect}
The connected component of a point $x \in X$ contains an $n$-proper point iff every clopen set containing $x$ contains an $n$-proper point.
\end{lemma}
\begin{proof}
The ``only if'' direction follows directly from Fact \ref{quasi}. For the ``if'' direction, let $C$ be the connected component of $x$ and
let $A$ be the set of $n$-proper points in $X$. The claim is trivial for $A = \emptyset$, so assume $A \neq \emptyset$. Note the $A$ is closed. 
Let $\{C_i\}_{i\in I}$ be the family of all clopen sets containing $x$. For any $i_1, \ldots i_n$, we have that $C_{i_1} \cap \cdots \cap C_{i_n}$ is 
clopen and contains $x$, so $C_{i_1} \cap \cdots \cap C_{i_n} \cap A \neq \emptyset$ by assumption. Thus the family
$\{ C_i \cap A \}$ has the Finite Intersection Property (see \cite[Theorem~3.1.1]{engelking}) which implies
\[
\emptyset \neq \bigcap_{i\in I} (C_i \cap A) = A \cap \bigcap_{i\in I} C_i = A \cap C.
\]
\end{proof}

\begin{lemma}
\label{prop}
Assume that $X$ is computably compact. Then
\[
\textrm{Prop}(n,m) := \left\{ \langle n, m \rangle ~:~ \textrm{the connected component of $x_n$ contains an $m$-proper point}  \right\} \in \Pi_{\langle n,m \rangle \mapsto 2m}^0
\]
\end{lemma}
\begin{proof}
As a consequence of Fact \ref{decomp}, there is an enumeration $(C_{n,k})_{n,k}$ such that $(C_{n,k})_k$ lists all clopen sets containing $x_n$, each as a finite union of basic balls. 
$C_{n,k}$ contains an $m$-proper point iff one of the open balls contains an $m$-proper point, or, equivalently, if the closure of one of the balls contains an
$m$-proper point (because $C_{n,k}$ as a closed set is also equal to the union of the closures of the balls). Lemma \ref{checkBall} thus yields
\[
\left\{ \langle n, m, k \rangle ~:~ C_{n,k} \textrm{ contains an $m$-proper point} \right\} \in \Pi_{\langle n,m,k \rangle \mapsto 2m}^0.
\]
By Lemma \ref{finiteIntersect}, the intersection of these sets over $k$ is equal to $\textrm{Prop}(n,m)$. The intersection does
not change the quantifier complexity. 
\end{proof}

\begin{lemma}
\label{proprank}
Assume that $X$ is computably compact. Then
\[
\textrm{PropRank} := \left\{ \langle n, m \rangle ~:~ \textrm{the connected component of $x_n$ has proper rank $m$}  \right\} \in \Sigma_{\langle n,m \rangle \mapsto 2m + 2}^0
\]
\end{lemma}
\begin{proof}
Follows from Lemma \ref{prop} because 
\[
\langle n, m \rangle \in \textrm{PropRank} \ \Leftrightarrow \ \langle n, m \rangle \in \textrm{Prop} \ \wedge \ \langle n, m + 1 \rangle \notin \textrm{Prop}
\]
\end{proof}

\begin{lemma}
\label{inner}
Assume that $X$ is computably compact.
\begin{itemize}
    \item[(i)] The distance $D(n)$ between a special point $x_n$ and the complement of its connected component is right-c.e.\ uniformly in $n$.
    \item[(ii)] $\textrm{Inner} := \left\{ n ~:~ x_n \textrm{ is in the interior of its connected component} \right\} \in \Sigma_2^0$
\end{itemize}
\end{lemma}
\begin{proof}
(i): As a consequence of Fact \ref{decomp} we can enumerate a sequence $(C_m)_m$ that contains (with repetition) all clopen sets $C_m$ with $x_n \in X \setminus C_m$, 
each as a finite union of basic balls, say
\[
C_m := \bigcup_{i = 1}^{k_m} B(x_{n_{m,i}}, r_{m,i}),
\]
and such that $r_{m,i} < 2^{-m}$. By Fact \ref{quasi}, we have $D(n) = d\left(x_n, \bigcup_m C_m\right)$. 
Define
\[
\tilde{D}(n) = \inf \left\{ d(x_{n_{m,i}}, x_n) ~:~ m \in \N,\ 1 \leq i \leq k_m \right\}.
\]
We claim $\tilde{D}(n) = D(n)$ which implies that $D(n)$ is right-c.e. 
In fact, for all $m \in \N$, $1 \leq i \leq k_m$, we have
\[
D(n) = d\left(x_n, \bigcup_\ell C_\ell \right) \leq d(x_n, C_m) \leq d(x_{n_{m,i}}, x_n),
\]
so $D(n) \leq \tilde{D}(n)$.
On the other hand, for every $\varepsilon > 0$ there is an $x \in \bigcup_\ell C_\ell$ such that $d(x, x_n) - D(n) < \varepsilon$. 
Choose $m$ such that $x \in C_m$ and $2^{-m} < \varepsilon$, then choose $i$ such that $x \in B(x_{n_{m,i}}, r_{m,i})$. 
Then 
\[
d(x_{n_{m,i}}, x_n) - D(n) \leq d(x_{n_{m,i}}, x) + d(x, x_n) - D(n) < d(x_{n_{m,i}}, x) + \varepsilon < r_{m,i} + \varepsilon < 2^{-m} + \varepsilon < 2\varepsilon.
\]
As $\varepsilon > 0$ was arbitrary, this yields $\tilde{D}(n) \leq D(n)$.

(ii): By (i) there is a computable $g: \N \times \N \to \Q$ such that $(g(n, \ell))_\ell$ converges to $D(n)$ from above. $x_n$ is in the interior
of its connected component iff $D(n) > 0$ iff $(\exists m)(\forall \ell)\ g(n,\ell) > 2^{-m}$. 
\end{proof}

A component of $X$ has nonempty interior iff its interior contains a special point, so
\[
m \in \rho(X) \ \Leftrightarrow \ (\exists n)\left[ n \in \textrm{Inner}\ \wedge\ \langle n, m \rangle \in \textrm{PropRank} \right]
\]
Proposition \ref{upperBound} then follows from Lemmas \ref{inner}(ii) and \ref{proprank}.

\section{Proof of Proposition \ref{lowerBound}}

\begin{lemma}
\label{1proper}
Uniformly relative to any oracle $O \subseteq \N$, there is a perfect c.e.\ closed set $G \subseteq [0,1]$ such that
$1 \in G$ and, considering the condition
\begin{equation}
\label{fef}
(\exists^\infty k)(\forall \ell)\ \langle k, \ell \rangle \in O,
\end{equation}
we have:
\begin{itemize}
    \item[(i)] If \eqref{fef} is true:
    \begin{itemize}
        \item[(a)] $G$ is a countably infinite union of disjoint closed intervals and $\{1\}$. 
        \item[(b)] The point $1$ is proper in $G$ with rank $1$.
        \item[(c)] There are no other points that are proper in $G$.
        \item[(d)] The connected component of $1$ is $\{1\}$ and has empty interior in $G$.
    \end{itemize}
    \item[(ii)] If \eqref{fef} is false, $G$ is a finite union of closed intervals 
        (and thus does not have any proper points). 
\end{itemize}
\end{lemma}
\begin{proof}
Define
\[
G_k =
\begin{cases}
\left[ 1 - 2^{-k}, 1 - 2^{-k} + 2^{-(k+2)} \right] & \text{if } (\forall \ell)\ \langle k, \ell \rangle \in O, \\
\left[ 1 - 2^{-k}, 1 - 2^{-(k+1)} \right]          & \text{otherwise.}
\end{cases}
\]
and $G := \bigcup_k G_k \cup \{1\}$. The $G_k$ are constructed as a sequence of closed intervals, closer and
closer to $1$, such that the gap between $G_k$ and $G_{k+1}$ is $2^{-(k+2)}$ if $(\forall \ell)\ \langle k, \ell \rangle \in O$;
the intervals intersect in $1 - 2^{-(k+1)}$ otherwise. 
This clearly implies (i)(a) and (ii). Also, $G$ is c.e.\ because the potential part of $G_k$ closing the gap to $G_{k+1}$ 
is added on the enumerable condition $(\exists \ell)\ \langle k, \ell \rangle \notin O$.

Assume \eqref{fef} is true and let $(\tilde{G}_j)_j$ be the maximal closed intervals contained in $G$, in ascending order.
(So each $\tilde{G}_j$ is a union of certain finitely many $G_k$.) Each $\tilde{G}_j$ is clopen in $G$. 
Every neighbourhood of $1$ contains infinitely many $\tilde{G}_j$, so $1$ is proper. 
The component of $1$ is $\{1\}$, because every $x \in G \setminus \{1\}$
is contained in one of the $\tilde{G}_j$ which is separated from $1$ by the gap between $\tilde{G}_j$ and $\tilde{G}_{j+1}$.
$\{1\}$ has empty interior because every neighbourhood $1$ also intersects some $\tilde{G}_j$ and thus exceeds $\{1\}$. 
So all properties listed under (i) hold.
\end{proof}

\begin{lemma}
\label{blocks}
Uniformly relative to any oracle $O \subseteq \N$, there is a computable sequence $(H_m)_{m \geq 1}$ of
perfect c.e.\ closed sets $H_m \subseteq [0,1]$ such that for every $m \geq 1$ 
\begin{itemize}
    \item[(i)] all components of $H_m$ with non-empty interior have proper rank $0$ in $H_m$,
    \item[(ii)] if $1$ is proper in $H_m$, its proper rank is at most $m$,
    \item[(iii)] $1$ is $m$-proper in $H_m$ iff 
    \[
    (\exists^\infty k_m) \cdots (\exists^\infty k_1)(\forall \ell)\ \langle m, k_m, \ldots k_1, \ell \rangle \in O^{(m)},
    \]
    where
    \[
    O^{(m)} := \left\{ \langle k_i, \ldots,k_1 \ell \rangle ~:~ \langle m, k_m, \ldots,k_1 \ell \rangle \in O  \right\}.
    \]
\end{itemize}
\end{lemma}
\begin{proof}
For fixed $m$, we use a recursive construction in $i = 1, \ldots, m$. In each stage we compute
a sequence of c.e.\ closed sets, each as a countable union of certain scaled and shifted sets from
the previous stage. 
In preparation, for every $i = 1, \ldots, m$ and $(k_m, \ldots, k_{i+1}) \in \N^{m-i}$, generalize
the definition of $O^{(m)}$ to
\[
O^{(m, k_m, \ldots, k_{i+1})} := \left\{ \langle k_m, \ldots,k_1 \ell \rangle ~:~ \langle m, k_m, \ldots,k_1 \ell \rangle \in O  \right\}.
\]

For $i=1$ compute sets $H^{(m, k_m, \ldots, k_2)}$ with the construction from Lemma \ref{1proper},
applied relative to $A^{(m, k_m, \ldots, k_2)}$, for each combination $(k_m, \ldots, k_2) \in \N^{m-1}$. 

For $i > 1$, assume $H^{(m, k_m, \ldots, k_{i})}$ have already been constructed for all $(k_m, \ldots, k_{i}) \in \N^{m-(i-1)}$. Define
\[
H^{(m, k_m, \ldots, k_{i+1})} := \bigcup_{k_i \in \N} \left(1 - 2^{k_i} + 2^{-(k_i+2)} H^{(m, k_m, \ldots, k_{i})} \right) \ \cup \ \{1\}.
\]

After the last stage of the construction, i.e.\ $i=m$, the set $H^{(m)}$ has been constructed. Put $H_m = H^{(m)}$. 
It is easy to show by induction on $i=1, \ldots, m$ that every set $H^{(m, k_m, \ldots, k_{i+1})}$ appearing in the construction fulfills
properties (i)-(iii) with $i$ in place of $m$ and $O^{(m, k_m, \ldots, k_{i+1})}$ in place of $O^{(m)}$, where the base 
case is yielded by Lemma \ref{1proper}. So (i)-(iii) hold for $H_m$ and $O^{(m)}$ in particular. 
Also, $H_m$ is c.e.\ as the union of a computable sequence of c.e.\ sets. 
\end{proof}

For the proof of Proposition \ref{lowerBound}, we need to construct a c.e.\ closed subset 
$K$ of $[0,1]$ such that $\rho(K) = A$ 
for given $A \in \Sigma_{m \mapsto 2m + 3}^0$ with $\{2n ~:~ n\in \N\} \subseteq A$.

By the characterization of $\Pi_{m \mapsto 2m + 3}^0$ mentioned in Section \ref{feiner}, there is a
computable predicate $p$ such that
\[
m \notin A \ \Leftrightarrow \ (\exists^\infty k_{m+1}) \cdots (\exists^\infty k_1)(\forall \ell)\ p(\langle m, k_{m+1}, \ldots k_1, \ell \rangle).
\]
Choose 
\[
O := \{ \langle m+1, k_{m+1}, \ldots k_1, \ell \rangle ~:~ p(\langle m, k_{m+1}, \ldots k_1, \ell \rangle) \}
\]
and note that $O$ is computable. Construct $(H_m)_m$ as in Lemma \ref{blocks} relative to this $O$. 
Then $(H_m)_m$ is a computable sequence of c.e.\ closed sets such that 
\begin{equation*}
\begin{array}{rl}
m \notin A
& \ \Leftrightarrow \ (\exists^\infty k_{m+1}) \cdots (\exists^\infty k_1)(\forall \ell)\ p(\langle m, k_{m+1}, \ldots k_1, \ell \rangle) \\
& \ \Leftrightarrow \ (\exists^\infty k_{m+1}) \cdots (\exists^\infty k_1)(\forall \ell)\ \langle m + 1, k_{m+1}, \ldots k_1, \ell \rangle \in O \\
&  \ \Leftrightarrow \ 1 \text{ is } (m+1) \text {-proper in } H_{m+1}.
\end{array}
\end{equation*}
Construct another sequence $(\tilde{H}_m)_m$ as in Lemma \ref{blocks}, but this time relative to $\N$, i.e.
$(\tilde{H}_m)_m$ is a generic sequence of c.e.\ closed sets such that $1$ is $m$-proper in each $\tilde{H}_m$.
For every $m$, define
\[
K_m :=
\begin{cases}
\tilde{H}_m \cup [1, 2]                    & \text{if $m$ is even}, \\
\tilde{H}_m \cup [1, 2] \cup (3 - H_{m+1}) & \text{if $m$ is odd.}
\end{cases}
\]
Each $K_m$ has $[1,2]$ as a connected component with non-empty interior and the $m$-proper point $1$. 
If $m$ is odd and $m \notin A$, the component $[1,2]$ also has the $(m+1)$-proper point $2$, so
the proper rank of $[1,2]$ is $m+1$.
If $m$ is odd and $m \in A$, the proper rank of $2$ is at most $m$, so the proper rank of $[1,2]$ is $m$.
All other components of $K_m$ with non-empty interior are the closed intervals from the constructions of
$\tilde{H}_m$ and $H_m$; they all have proper rank $0$. 
To summarize:
\[
\rho(K_m) :=
\begin{cases}
\{0, m+1\}                    & \text{if $m$ is odd and $m \notin A$}, \\
\{0, m\}                      & \text{otherwise.}
\end{cases}
\]
For the construction of $K$ it is sufficient to line up scaled down separated copies of the $K_m$ in $[0,1]$
and add the closing point at $1$:
\[
K = \bigcup_m \left(1 - 2^{-m} + 2^{-(m+3)} + 2^{-(m+4)}K_m \right) \cup \{ 1 \}.
\]
Then
\[
\begin{array}{rl}
\rho(K) 
& = \bigcup_m \rho(K_{2m}) \cup \bigcup_m \rho(K_{2m+1}) \\
& = \{ 2m ~:~ m \in \N \} \cup \left( \{ 2m+1 ~:~ m\in \N,\, 2m+1 \in A \} \cup \{ 2m+2 ~:~ m\in \N,\, 2m+1 \notin A \} \cup \{ 0 \} \right) \\
& = \{ 2m ~:~ m \in \N \} \cup \{ 2m+1 ~:~ m\in \N,\, 2m+1 \in A \} \\
& = A
\end{array}
\]
where the last equality follows from the assumption $\{ 2m ~:~ m \in \N \} \subseteq A$.

\bibliographystyle{plain}

\begin{thebibliography}{9}

\bibitem{ashknight}
C.J. Ash and J. Knight.
\textit{Computable Structure Theory and the Hyperarithmetical Hierarchy}.
North Holland Publishing Co., 2000.

\bibitem{stone}
N. Bazhenov, M. Harrison-Trainor, and A. Melnikov.
\textit{Computable Stone Spaces}.
Annals of Pure and Applied Logic, vol.~174 (2023), no.~9, p. 103304.

\bibitem{bazhenov}
N. Bazhenov, A. Melnikov, and K. M. Ng.
\textit{Every $\Delta_2^0$ Polish space is computable topological}.
Proc. Amer. Math. Soc., vol.~152 (2024), pp. 3123--3136.

\bibitem{benshahar}
S. Ben-Shahar and H. T. Koh.
\textit{Comparing Notions of Presentability in Polish Spaces and Polish Groups}.
Annals of Pure and Applied Logic, vol.~176 (2025), no.~5, p. 103564.

\bibitem{bosshert}
 V. Bosserhoff and P. Hertling.
\textit{Effective Subsets Under Homeomorphisms of $\R^n$}.
Information and Computation, vol.~245 (2015), pp. 197--212.

\bibitem{brattkapresser}
 V. Brattka and G. Presser.
\textit{Computability on Subsets of Metric Spaces}.
Theoretical Computer Science, vol.~305 (2003), pp. 43--76.

\bibitem{compcomp}
R. Downey and A. Melnikov.
\textit{Computably Compact Metric Spaces}.
The Bulletin of Symbolic Logic, vol.~29 (2023), no.~2, pp.~170--263.

\bibitem{compstruc}
R. Downey and A. Melnikov.
\textit{Computable Structure Theory -- A Unified Approach}.
Springer, 2025.

\bibitem{engelking}
R. Engelking.
\textit{General Topology}.
Heldermann, 1989.

\bibitem{feiner}
L. Feiner.
\textit{Hierarchies of Boolean Algebras}.
The Journal of Symbolic Logic, vol.~35 (1970), no.~3, pp.~365--374.

\bibitem{harrison}
M. Harrison-Trainor, A. Melnikov, and K. M. Ng.
\textit{Computability of Polish Spaces up to Homeomorphism}.
The Journal of Symbolic Logic, vol.~85 (2020), no.~4, pp.~1664--1686.

\bibitem{spectra}
M. Hoyrup, T. Kihara, and V. Selivanov.
\textit{Degree Spectra of Homeomorphism Type of Compact Polish Spaces}.
The Journal of Symbolic Logic, Published online 2023:1-32. doi:10.1017/jsl.2023.93

\bibitem{counterex}
H.T. Koh , A.G. Melnikov, and K.M. Ng.
\textit{Counterexamples in Effective Topology}.
The Journal of Symbolic Logic, Published online 2025:1-24. doi:10.1017/jsl.2025.25

\bibitem{abelian}
M. Lupini, A. Melnikov, and A. Nies.
\textit{Computable Topological Abelian Groups}.
Journal of Algebra, vol.~615 (2023), pp.~278--327.

\bibitem{separating}
A.G. Melnikov and K.M. Ng.
\textit{Separating Notions in Computable Topology}.
International Journal of Algebra and Computation, vol.~33 (2023), no.~8, pp.~1687--1711.

\bibitem{rogers}
H. Rogers.
\textit{Theory of Recursive Functions and Effective Computability}.
McGraw-Hill, 1967.

\end{thebibliography}

\end{document}